\documentclass[11pt, a4paper]{article}
 \usepackage{geometry}
\usepackage[english]{babel}

\usepackage{amsmath,amssymb,amsthm,amsthm}
\usepackage{mathrsfs}
\usepackage{enumerate}

\usepackage{todonotes}  

\usepackage[noadjust]{cite} 

\theoremstyle{plain}
\newtheorem{theorem}{Theorem}[section]
\newtheorem*{theorem*}{Theorem} 
\newtheorem{proposition}[theorem]{Proposition}
\newtheorem{lemma}[theorem]{Lemma}

\theoremstyle{definition}
\newtheorem{definition}[theorem]{Definition}
\theoremstyle{definition}

\theoremstyle{definition}
\newtheorem*{definition*}{Definition}

\theoremstyle{definition}

\theoremstyle{definition}
\newtheorem*{notation*}{Notation}

\theoremstyle{definition}
\newtheorem{remark}[theorem]{Remark}
\theoremstyle{definition}
\newtheorem*{remark*}{Remark}
\theoremstyle{definition}

\theoremstyle{definition}
\newtheorem*{remarks*}{Remarks}

\theoremstyle{remark}
\theoremstyle{remark}

\newtheorem*{example*}{Example}
\newtheorem*{examples*}{Examples}

\theoremstyle{definition}

\theoremstyle{remark}

\newcommand{\N}{{\mathbb N}}
\newcommand{\R}{{\mathbb R}}

\def\N{\mathbb{N}}

\newcommand{\RR}{{\R}}
\newcommand{\NN}{{\N}}
\newcommand{\EE}{{\E}}

\newcommand{\PP}{{\mathbb P}}

\newcommand{\Law}{{\textrm{Law}}}

\newcommand{\KL}{{\mathcal H}} 
\newcommand{\WW}{{\mathbb W}}

\usepackage{amsmath}
\usepackage{amsfonts}
\usepackage{amsthm}
\usepackage{amssymb}

 \newcommand{\E}{\mathbb{E}}

\makeindex

\usepackage{tikz}
\usetikzlibrary{matrix,arrows.meta}
\usepackage{tikz-cd} 
\usetikzlibrary{arrows}
\tikzcdset{arrow style=tikz, diagrams={>=stealth}} 
\usepackage{graphicx}  

\usetikzlibrary{shapes.geometric,quotes}

\newcommand{\dd}{\mathrm{d}}

\makeatletter
\newcommand\RedeclareMathOperator{%
  \@ifstar{\def\rmo@s{m}\rmo@redeclare}{\def\rmo@s{o}\rmo@redeclare}%
}
\newcommand\rmo@redeclare[2]{%
  \begingroup \escapechar\m@ne\xdef\@gtempa{{\string#1}}\endgroup
  \expandafter\@ifundefined\@gtempa
     {\@latex@error{\noexpand#1undefined}\@ehc}%
     \relax
  \expandafter\rmo@declmathop\rmo@s{#1}{#2}}
\newcommand\rmo@declmathop[3]{%
  \DeclareRobustCommand{#2}{\qopname\newmcodes@#1{#3}}%
}
\@onlypreamble\RedeclareMathOperator
\makeatother

\RedeclareMathOperator*{\div}{{div}}

\newcommand{\deq}{{\;\stackrel{\mathrm{def}}{=}\;}}

\usepackage{hyperref}

\title{A criterion for proving entropy chaos on path space}

\author{Luigi Borasi, Francesco Carlo De Vecchi and Stefania Ugolini}

\begin{document}
  
\maketitle

\abstract{A criterion for proving a strong form of propagation of chaos on the path space,  known as entropy chaos,  for a general interacting diffusion system is proposed. Our analysis focuses on the class of conservative diffusions introduced by Carlen, which are characterized by infinitesimal characteristic pairs, that is, 
a time-marginal probability density and a current velocity field. A key property of this broad class is that the processes remain diffusions under time-reversal. We prove that, given a suitable bound on the relative entropy (with respect to the Wiener measure) and the weak convergence of both drifts and fixed-time marginal densities, strong entropy chaos at the process level is achieved in the infinite particle limit, provided the limit drift satisfies a specific regularity condition. This stochastic framework encompasses various singular interacting particle systems and their related asymptotic scenarios.
}

\medskip


\section{Introduction}
By a $d-$dimensional \textit{conservative diffusion} we mean  the weak solution to a 
$d-$dimensional 
stochastic differential equation (SDE) given by
\begin{equation}\label{Conservative_diffusions_intro}
\dd X(t)=b(X(t),t)\dd t+\dd W(t),
\end{equation}
with $W(t)$ a $d-$dimensional Brownian motion,
where the drift is expressed as the sum of two different contributions
\begin{equation}\label{eq:2}
    b(x,t)=\frac{\nabla \rho}{2\rho}(x,t)+v(x,t),\quad x\in\RR^d,
\end{equation}
with $\rho$ being the probability density  of the law of the process $X$ and $v$ a vector field.  The \textit{infinitesimal characteristics} pair $ (\rho(t),v(t))$  is assumed to satisfy both the \emph{finite energy condition} 
\begin{equation}\label{finite_condition}
    \int_0^T\int_{\RR^d} (|u|^2+|v|^2)(x,t)\rho(x,t)dxdt < \infty
    ,
\end{equation}
and, in a weak sense, the \textit{continuity equation}
(also referred to as \textit{law of (mass) conservation})
\begin{equation}\label{continuity_equation_intro}
 \partial_t \rho+\nabla \cdot (\rho v)=0.
\end{equation} 
 The logarithmic gradient of the probability density 
 \begin{equation}\label{osmotic_volocity}
     u(x,t)=\frac{\nabla \rho(x,t)}{2\rho(x,t)}
 \end{equation}
takes the name of \textit{osmotic velocity}, 
while the vector field $v$ of \textit{current velocity}. 
Under the above assumptions, specifically \eqref{finite_condition} and \eqref{continuity_equation_intro}, the SDE \eqref{Conservative_diffusions_intro} has a unique weak solution, its law admits a probability density given by $\rho(t),$ and the solution process solves also a  time reversed SDE (cf.\ Theorem \ref{Carlentheorem}). 
We recall that the conservative diffusion class,  formally introduced by Nelson (\cite{NelsonQF}) in the mid-1980s, promoted the study of the related and fundamental notions of Nelson forward and backward derivatives, 
the introduction of an integration by parts formula (\cite{Cattiaux2021TimeRO}),
and, most notably, the investigation of the property of time reversal of diffusion processes (\cite{haussmann1986time,millet1989integration}), in the sense that a diffusion process remains a diffusion under time-reversal transformation.
More recently, the time reversal approach  under finite entropy condition has been generalized to the case of jump processes \cite{CONFORTI202285}. 
The behaviour under time reversal of 
diffusion processes has proven to be useful, for example, for providing a probabilistic representation of the solution of a Fokker-Planck type partial differential equations with prescribed final data \cite{izydorczyk:hal-02902615}.
Finally, we recall that the conservative diffusion class have a stochastic variational formulation obtained, in the context of Nelson stochastic mechanics, from a Guerra-Morato Lagrangian (\cite{GuMo}), or from other stochastic variational principles (\cite{pavon1995hamilton,pavon1989stochastic,Yasue}).

\medskip

We consider a system of $Nd-$dimensional diffusions
\begin{equation}\label{eq:6}
     \dd X_N(t)=b_N(X_N(t),t)\dd t + \dd W_N(t)
     ,
\end{equation}
where $W_N$ is an $Nd-$dimensional Brownian motion and where
$b_N$ is an $Nd-$valued function satisfying 
the following $L^2-$condition
\begin{equation}\label{L2_condition_intro}
    \mathbb{E}_{\mathbb{P}_N}\left[\int_0^T b_N^2(t)\dd t\right]=\int_0^T\int_{\R^{Nd}}b^2_N(x_N,t) \rho_N(x_N,t)\dd x_N \dd t  < \infty
    ,
\end{equation}
where we denote by $\mathbb E_{\mathbb P_N}$
the expectation with respect to the joint-law $\mathbb P_N$ on the path space of the 
$Nd-$valued process $X_N$.
This condition is equivalent to a finite relative entropy condition 
by Girsanov theory (as pointed out in \cite{Follmer_Wiener_space}). 
The relative entropy is here considered between the diffusion laws $\mathbb{P}_N$ and the Wiener one. We will stress that \eqref{L2_condition_intro} implies both the existence for every $t\in [0,T]$ of a quite regular probability density $\rho_N(t)$ and the well-posedness of the osmotic velocity $u_N(t)$ in \eqref{osmotic_volocity} (cf.\ Proposition \ref{L2drift_consequences}). The latter is also known as Fisher information and has very good properties, including convexity and superadditivity (see, e.g., \cite{carlen1991superadditivity,HauMis}). \\ 

In \cite{Carlen2010} the authors introduced a new notion of chaotic behavior for Kac-Boltzmann model. This property can be proved by an entropy approach, 
and, in general, provides a stronger form of chaoticity, denominated entropy chaos,  specifically in total variation or $L^1$ norm. 
This strong convergence is to be compared to the usual weak convergence involved in the standard propagation of chaos or Kac chaos. These entropy chaos concepts where initially introduced for the fixed-time marginal densities of the process law.
We refer also to \cite{HauMis} for more details.\\ 

Our final goal is to investigate propagation of chaos type results of the 
joint law $\mathbb P_N$ as $N\rightarrow\infty$ at the process level.
In particular we discuss a criterion to prove strong $\mathbb P$-entropy chaos \textit{on the path space} (cf.~Defined~\ref{chaoticity_defiition} below).
Here $\mathbb P$ is the law on the path space of a $d-$dimensional conservative diffusion
as in \eqref{Conservative_diffusions_intro}.
Roughly speaking, we give a criterion to show that, in the limit $N\rightarrow\infty$,
the $k$-marginals of the joint law $\mathbb P_N$ 
converge in total variation
to the $k$-th product of copies of the law $\mathbb P$.
This criterion has a wide range of applicability.
Indeed very weak assumptions are imposed on the regularity of the drift coefficients 
$b_N$, notably only the condition \eqref{finite_condition}.
Moreover, the drift coefficients $b_N$ 
can exhibit any dependence on the $Nd-$components of the process $X_N$
(in particular we do not assume $b_N$ to be a sum
of pairwise interactions).\\

 The proofs of the main convergence theorems we discuss here, Theorem \ref{Th:Criterium} and Theorem \ref{Strong_entropic_chaos}, are based on some notable properties of the relative entropy between the involved probability measures.
 In particular we take advantage of some recent general results on the description of the stochastic optimal transport problem (\cite{DeVecchi_Rigoni2024})  
 in a form  fine-tuned in \cite{BEC2025} to our specific convergence criterion. 
 See Theorem \ref{Th:Criterium}.\\ 

The first theorem (Theorem \ref{Th:Criterium}) assumes a control of the sequence of 
relative entropy w.r.t. a given fixed reference measure, 
the weak convergence of the fixed time marginal densities for any $t\in [0,T]$,
and, finally, the weak convergence of the drifts \cite{BEC2025,DeVecchi_Rigoni2024}.
A suitable representation of the relative entropy between the law of a conservative diffusion and the Wiener measure in terms of the finite energy condition \eqref{finite_condition} and the difference of the Boltzmann entropy evaluated at the initial and final time, proved in \cite{BEC2025}, \cite{Cattiaux2021TimeRO}, is also used.\\

A posteriori, Theorem \ref{Th:Criterium} provides the convergence of the 
relative entropy sequence to the limit relative entropy.
We denominate this convergence \textit{weak entropy chaos}
(cf.~\ref{chaoticity_defiition}).
 We stress the advantage of introducing the above notion of weak entropy chaos w.r.t.\ to a reference measure, specifically the Wiener measure.
 The second theorem (Theorem~\ref{Strong_entropic_chaos}), 
 which is a consequence of the first one, constitutes our core criterion and provides minimal regularity conditions which are necessary to apply the first theorem in order to achieve 
 the strong entropy chaos on path space. 
 These conditions regards, significantly, only the limit drift $b$.\\

The novelty is therefore that we introduce the notion of weak-entropy chaos on the path space
and show how it allows to prove the strong-entropy chaos
via the aforementioned criterion.
Moreover, as we have already said, we work with
laws on the path space, while the classical Kac and entropy chaos are usually discussed for the marginal of the laws at a fixed time. Finally, in the present paper, we propose to consider the class of conservative diffusions, seen not as a particular case but as a very general diffusion class. Indeed,  the only regularity assumption is that the drifts satisfies the $L^2$ condition \eqref{L2_condition_intro}.
In particular, we judge the time-reversal diffusion property
as a ``universal'' property satisfied
by a very large class of processes
even if it is not a-priory assumed (see Section 4.1).
By F\"ollmer results (cf.~Theorem \ref{Follmer_theorem} below)
this time-reversal property is a natural consequence of the relative entropy invariance for the time-reversal transformation. 
Therefore, we are lead to believe that
  conservative diffusions 
will play a central role in future research on the subject.\\

The problem of providing quantitative estimate for entropic chaos is not discussed in the present paper, even though this research direction is actually  very active. We mention \cite{Jabin2018},\cite{Lacker_2023} as examples of quantitative results in this context, obtained, for the fixed time marginal densities, 
via large deviation techniques.

\medskip
 
The plan of the present paper is the following. In Section 2 we introduce the class of conservative diffusions. We provide a very brief summary of the entropy approach to propagation of chaos in Section 3. 
The criterion for establishing the entropy chaos on path space is described in Section 4. Finally, some examples of applications of our strategy are discussed in Section 5.

\section{Conservative diffusion processes}

\medskip
\subsection{Time reversal of diffusion processes}
In 1984 Nelson \cite{NelsonQF} introduces a new class of diffusion processes, denominated \textit{conservative diffusions},  which are solutions to Brownian-driven SDEs with gradient drift of McKean-Vlasov type and constant diffusion coefficient adapted to describe the kinematic of  quantum particles. 
These particular diffusion processes, similarly to quantum mechanics, are quasi-invariant (in the sense of mutual absolute continuity of their probability laws on the path space) under time reversal.\\

After the formal definition, 
the rigorous characterization and construction of conservative diffusions  
was given by Carlen \cite{carlen1984conservative}. From the probabilistic point of view, 
the cited publication leads to the formulation of
the problem of when a given diffusion remains a diffusion under time reversal. 
The main contributions to this general question are due to Haussmann and Pardoux (1986),  Millet, Nualart and Sanz (1999) and Follmer (1986) and(\cite{Follmer_Wiener_space,haussmann1986time,millet1989integration}). 
See also the pioneering results by Anderson \cite{ANDERSON1982313}. 
We briefly recall here the main steps of the first  cited reference. 
We restrict ourselves to the special case of constant diffusion coefficient for simplicity and coherence with the rest of the present paper. Subsequently, the approaches by F\"ollmer and by Carlen are illustrated with more details.\\

Let us start with the general isotropic time-dependent SDE on $[0,T]$,
\begin{equation}\label{generalSDE}
    \dd X(t)=b(X(t),t)\dd t+\dd W(t), \quad t\in [0,T].
\end{equation}

The main fundamental hypothesis, which is common to all the cited approaches, is the existence of a probability density $ \rho_t(x),$  with a suitable integrability property, for the law of the process $X$. Let us introduce the reversed process 
\begin{equation}\label{reversed_process}
 Y(t)=X({T-t}),\quad  t\in [0,T],
\end{equation}
 and the reversed filtration
\begin{equation}\label{reversed_filtration}
   \tilde{\mathcal{F}}_t=\sigma \{Y(s): s\leq t\}= \sigma \{X(u): T-t \leq u\leq T\}.
\end{equation}

Under classical regularity conditions, specifically that the drift $b$ is locally Lipschitz continuous with linear growth, and the law of $X_t$ admits a density $\rho_t(x)$ such that for any open set $D$
    \begin{equation}\label{density_condition}
        \int_0^T\int_D \left(|\rho_t(x)|^2+\sum_j |\nabla_{x_j}\rho_t(x)|^2 \right) \dd x\dd t <+\infty,
    \end{equation} 
one proves, by using classical Ito calculus (cf.~\cite{haussmann1986time}), 
that $Y$ is a $\tilde{\mathcal{F}}_t-$semi-martingale solving the martingale problem corresponding to a specific generator.
More precisely, the reversed process $Y$ satisfies the following SDE
    \begin{equation}\label{general_reversed_SDE}
    \dd Y(t)=\tilde{b}(Y(t),t)\dd t+\dd \tilde{W}(t), \quad t\in [0,T],
\end{equation}
where
    \begin{equation}
     \tilde{b} (x,t)=-b(x,T-t)+\nabla \log \rho(x,T-t),  
    \end{equation}
    and $\tilde{W} $ is an $\tilde{\mathcal{F}}_t-$Brownian motion. \\

The key point is that the martingale part in \eqref{generalSDE} has a non trivial behavior under $\tilde{\mathcal{F}}_t,$ because under time-reversal the adaptability and martingale properties are not conserved. Indeed, introducing the reversed Brownian motion
\begin{equation}\label{reversed_increment}
    B(t)=W(T)-W({T-t}),
\end{equation}
 we can immediately realize that it is not an $\tilde{\mathcal{F}}_t-$Brownian motion since 
 $\tilde{\mathcal{F}}_t-$contains the future of the forward process $X$, that is
 \begin{equation}
    \mathbb{E}[B({t+h})-B({t})|\tilde{\mathcal{F}}_t]= \mathbb{E}[W({T-t})-W({T-t-h})|\tilde{\mathcal{F}}_t]\neq 0.
 \end{equation}
By using Markov property, which is an invariant property under time reversal, one gets
\begin{align*}
\mathbb{E}[W({T-t})-W({T-t-h})|\tilde{\mathcal{F}}_t]&=\mathbb{E}[W({T-t})-W({T-t-h})|X({T-t})]\\&=\mathbb{E}\left[\left.\left(X({T-t})-X({T-t-h})-\int_{T-t-h}^{T-t} b(X(s),s)\dd s\right)\right|Y(t)\right],
\end{align*}
where the Brownian increment has been expressed in term of the process $X$ via \eqref{generalSDE}. By taking advantage of the definition of the (backward) drift (see Proposition~\ref{Follmer_theorem}
below),
and the continuity of the drift coefficient, they are able to show that  the 
$\tilde{\mathcal F}_t$-predictable part $\Gamma$ of the process $B$ has necessarily the following form
\begin{equation}
    \Gamma(s)=b(Y(s),T-s)-\tilde{b}(Y(s),s).
\end{equation}
Thus, by defining 
$$
\tilde{W}(t)=B(t)-\int_0^t \Gamma(s) \dd s,
$$ 
we get that $\tilde{W}$ is a $\tilde{\mathcal{F}}_t-$continuous martingale. Since its quadratic variation is the same as the quadratic variation of $B$, by Levy characterization theorem, it is an $\tilde{\mathcal{F}}_t-$Brownian motion. In \cite{millet1989integration} similar results are obtained by using the Malliavin calculus.

\subsection{An entropy approach to time reversal}
F\"ollmer, in \cite{Follmer_Wiener_space}, gives a different approach to conservative diffusions.
The difference lies in the fact that his approach makes fundamental use of the relative entropy, between the law of the process $X$ and the Wiener law, and Girsanov theorem, taking advantage of the key fact that the relative entropy is invariant under time reversal transformation.
F\"ollmer considers time reversal of a SDE with constant diffusion coefficient and general non Markovian drift coefficient. Here we recall only the main steps for the Markovian case.  \\

 Concretely, he starts from a reference measure on the path space $C([0,T])$, specifically the Wiener measure $\mathbb{W},$   and then he considers a generic probability measure $\mathbb{P}$ with the property to be  absolutely continuous  with respect to the reference measure. By Girsanov theory this fact is equivalent (\cite{Follmer_Wiener_space}, Proposition 2.11) to saying that there exists an adapted process $b(X(t),t)$ satisfying the condition
\begin{equation}\label{L2_condition}
    \mathbb{E}_{\mathbb{P}}\left[\int_0^T b(t)^2\dd t\right]< \infty
    ,
\end{equation}
and such that, under $\mathbb{P}$, the canonical process $X$ weakly solves a SDE with drift $b$. Condition \eqref{L2_condition} will be fundamental in all the rest of the present paper.\\ 

For the convenience of the reader, we collect in the following theorem the main results  in \cite{Follmer_Wiener_space}. 
Let us denote with $R$ the operation of time-reversal acting on the path-space, i.e. $R\omega(t)=\omega(T-t), \quad t\in [0,T]$,
where $\omega$ denotes a path of the process.

\begin{theorem}[F\"ollmer theorem]\label{Follmer_theorem}
Assume $\mathbb P$ is the law of a Markov diffusion process and $\mathbb W$ the Wiener measure. Assuming \eqref{L2_condition}, then the following properties hold.
\begin{enumerate}
    \item The time reversed $\mathbb{\tilde{P}}=\mathbb{P}\circ R$ is such that $H(\mathbb{\tilde{P}}|\mathbb{W})<\infty, $ i.e. there exists and adapted process $\tilde{b}$ such that 
\begin{equation}
    \tilde{W}(t)=X(t)-X(0)-\int_0^t \tilde{b}(X(s),s) \dd s, \quad t\in [0,T],
\end{equation}
is a Wiener process under $\mathbb{\tilde{P}}$ and, for any $t < T,$ it holds that 
$$
   \mathbb{E}_{\mathbb{\tilde{P}}}\left[\int_0^T \tilde{b}(t)^2\dd t\right]< \infty. 
$$
\item For almost all $t\in [0,T]$, in $L^2(\mathbb P)$
    \[
{b}(t)=\lim_{h \downarrow 0}\mathbb{E}\left[\left.\frac{X({t+h})-X(t)}{h}\right|\mathcal{F}_t\right]
.
\]
\item For almost all $t\in [0,T]$, in $L^2(\mathbb P)$
    \[
\tilde{b}(t)=\lim_{h \downarrow 0}\tilde{\mathbb{E}}\left[\left.\frac{X({t+h})-X(t)}{h}\right|\mathcal{F}_t\right]
=\lim_{h \downarrow 0}\mathbb{E}\left[\left.\frac{X({T-(t+h))}-X({T-t})}{h}\right|\mathcal{\tilde{F}}_{T-t}\right]
,
\]
where $\tilde{\mathbb E}$ denotes the expectation with respect to the measure $\tilde{\mathbb P}$.
\item For almost all $t\in[0,T]$, there exists the density $\rho_t$ which is weakly differentiable.
    \item The following duality equality holds $$\tilde{b}(x,T-t)+b(x,t)={\nabla \log \rho(x,t)}.$$
    \item The osmotic velocity satisfies the following finite condition 
    $\mathbb{E}\left[\int_t^T|\frac{\nabla \rho_s}{\rho_s}|^2(X_s)\dd s\right]<\infty$.
\end{enumerate}
\end{theorem}
\begin{proof}
    See \cite{Follmer_Wiener_space}.
\end{proof}

For a recent deep generalization of the above F\"ollmer framework see \cite{Cattiaux2021TimeRO}. 

\begin{remark}
    
Let us provide a few comments of the results in Theorem \ref{Follmer_theorem}.
 In accord with Nelson \cite{NelsonQF},\cite{NelsonCD}, under the hypothesis \eqref{L2_condition}, interestingly, the forward drift can be computed as a stochastic forward derivative under $\mathbb{P}$ in $L^2$ (see point $2.$ in the theorem) and
the backward drift can be computed equivalently either as stochastic forward derivative under $\mathbb{\tilde{P}}$ or as a stochastic backward derivative under $\mathbb P$ (see point $3.$ of the previous theorem). \\

Since $\mathbb{P}$ is absolutely continuous with respect to the Wiener measure $\mathbb{W}$, and $\mathbb{W}(t)$ admits  a smooth density with respect to the Lebesgue measure, then also the distribution of the process $X_t$ under $\mathbb{P}$ is absolutely continuous and admits a density $\rho_t$  which satisfies the regularity property at point $4.$ in the theorem. The duality equation given in point $5.$ (in the Markovian case) of the theorem is the same as in \cite{haussmann1986time,{millet1989integration}}. Finally the last result in the theorem is a finite condition on the osmotic velocity, or Fisher information. See, e.g., \cite{HauMis}.
\end{remark}

\subsection{Carlen diffusion class: definition and well-posedness}
 With the aim to prove the well-posedness of the class of diffusion introduced by Nelson, 
 with the original purpose of describing the kinematics of a quantum particle, Carlen  (\cite{CarlenCD,Carlen:162610}) rigorously defines
the class of conservative diffusions 
in terms of characteristics $(\rho,v)$. 
These diffusions are characterized by  the \textit{proper infinitesimal
  characteristics} $(\rho(t),v(t))$
consisting of a time-dependent probability density $\rho(x,t)$ and a
time-dependent vector field $v(x,t)$ defined $\rho(x,t)dxdt$-a.e. 
As described above, these diffusions possess a time-reversal symmetry, in the sense that 
the process is a diffusion both forward and backward in time
(see \cite{CarlenCD,NelsonQF}). The definition of the Carlen class, which assumes minimal regularity conditions
 for the existence of the related diffusion process, reads properly as follows.
\begin{definition}(see, e.g., \cite{CarlenCD,Carlen:162610}) \label{conservativediffusiondefinition}   
  The space of \emph{proper infinitesimal characteristics} is the set of pairs $(\rho(t),v(t))$,
  with $\rho: \RR^d \times \RR_+\rightarrow \RR_+$, $v:\RR^d\times \RR_+\rightarrow\RR^d$ functions
  satisfying the following properties:
  a \textit{continuity equation} in weak form,
  \begin{equation}\label{eq:38} 
    \int_{\R^d}f(x,T)\rho(x,T)dx-\int_{\R^d}f(x,0)\rho(x,0)dx=\int_0^T\int_{\R^d}(v(t)\cdot\nabla
    f)\rho(x,t)dx,
  \end{equation}
  and a \textit{finite energy condition},
  \begin{equation}\label{eq:39}
    \int_0^T\int_{\R^d} (|u(x,t)|^2+|v(x,t)|^2){\rho}(x,t)\dd x\dd t
    <+\infty,
  \end{equation}
  \noindent  with $u(x,t)=\frac{\nabla \rho(x,t)}{2\rho(x,t)}$, for all $T\geq 0$ and all $f\in
  C_0^{\infty}(\R^{d+1})$.
\end{definition}
\begin{remark}
  The equation \eqref{eq:38} is a weak form of the continuity equation
  $\partial_t\rho + \nabla\cdot(\rho v)=0$.  
\end{remark}

The definition of conservative diffusion has been generalized in the framework of admissible flows of probability measure in \cite{cattiaux2024entropypathspaceapplication}, Definition 2.6.

The following result states that
to any pair of proper infinitesimal characteristics there is an associated well-defined diffusion process. 

\begin{theorem}[Carlen theorem]\label{Carlentheorem} 
Consider $(\Omega, \mathcal F, \mathcal F_t,X_t) $ with
$\Omega=C(\R_+,\R^d)$, $X$ the canonical process
$X_t(\omega)=\omega(t)$, $\omega\in\Omega$, and $\mathcal F_t=\sigma(X_s, s\leq t)$
its natural filtration.
  If $(\rho,v)$ is a proper infinitesimal characteristics  pair in the sense of Definition \ref{conservativediffusiondefinition}, putting 
  \begin{equation}\label{eq:40} 
    b\deq u+v, \quad 
    u(x,t)\deq
    \begin{cases}
      \frac{\nabla\rho(x,t)}{2\rho(x,t)}, & \rho(x,t)\ne 0,\\
      0, & \rho(x,t)=0,
    \end{cases}
  \end{equation}
  then there exists a unique Borel
  probability measure $\mathbb P$ on $\Omega$ such that
  \begin{enumerate}[(i)]
  \item $(\Omega, \mathcal F, \mathcal F_t,X(t),\mathbb P) $ is a Markov
    process;

  \item the image of $\mathbb P$ under $X(t)$ has a density
    $\rho(x,t)$;

  \item $W(t):=X(t)-X(0)-\int_0^tb(X(s),s)ds$ is a $(\mathbb P,\mathcal F_t)$-Brownian motion, which means that the following Brownian-driven SDE, 
    \begin{equation}\label{eq:41}
      \dd X(t) = b(X(t),t)\dd t + \dd W(t),
    \end{equation}
    admits a weak solution.
      \end{enumerate}
\end{theorem}
\begin{proof}
    See \cite{CarlenCD,Carlen:162610,carlen1984conservative}. 
\end{proof}

\begin{remark}
    In Carlen's work, the diffusion process is constructed by starting from given marginal densities using analytic techniques. In \cite{Cattiaux1996} the  approach has been generalized and employed to find a diffusion process which minimizes the relative entropy w.r.t.\ a reference diffusion measure and subject to fixed marginal distributions.
\end{remark}

\section{Boltzmann entropy, relative entropy, entropic chaos versus Kac chaos}
Let us denote by $\mathcal{P}(E) $ the set of probability measures $\mathbb{P}$ on a Polish space $E,$ equipped  with its Borel $\sigma-$field. We first introduce here the Boltzmann entropy for probability measures on $E=\R^d$  admitting densities with respect to the Lebesgue measure. A regularity condition on the density is necessary for its well-posedness. See \cite{HauMis} for more details.
\begin{definition}[Boltzmann entropy]\label{def_boltzmann}
  Given a measure $\mu$ on $E=\R^d$ with density $\rho$ with respect to the Lebesgue measure on $E$, if $\rho$ has a finite moment of order $k>0$, then
  the Boltzmann entropy $H(\mu)\equiv H(\rho)$ is defined as
  \begin{align*}
    H(\rho)\deq \int_{\RR^k} \rho(x)\log(\rho(x))\dd x
    .
  \end{align*}
\end{definition}
\begin{remark}
    The functional $H:\mathcal{P}(E)\to \R$ is convex and lower semi-continuous with respect to the weak convergence of probability measures. See, e.g., \cite{HauMis}.
\end{remark}
The relative entropy between two probability measures, differently from the Boltzmann entropy, is always well-defined.

\begin{definition}
    Given $\mu, \nu \in \mathcal{P}(E), $ the relative entropy between the two measures is defined, when $\mu$ is absolutely continuous with respect to $\nu$ as
    \begin{align*}
    H(\mu|\nu)\deq \int_{\RR^d}\log\left( \frac{\dd \mu}{\dd \nu}\right)\dd \mu
    ,
  \end{align*}
  where $\frac{\dd\mu}{\dd\nu} $ is the Radon-Nikodym derivative,
  and $H(\mu|\nu)\deq\infty$
  when $\mu$ is not absolutely continuous with respect to $\nu$.
    
\end{definition}

Also the relative entropy is lower semi-continuous and, even thought it is not a distance, it controls the total variation (TV) distance thank to the well-known Csiszar-Kullback-Leibler-Pinsker inequality
\begin{equation}\label{CKLP_inequality}
    ||\mu-\nu||_{TV}^2 \leq 2H(\mu|\nu).
\end{equation}

 The following lemma expresses the relative entropy between the law of a conservative diffusion and the Wiener measure in terms of the (Carlen) finite energy condition \eqref{finite_condition} and of the Boltzmann entropy evaluated at the initial and final time. 
\begin{lemma}\label{lem:vu}  
Let $\PP$ denote the law of a conservative diffusion on $E$ which is the weak solution to a isotropic Brownian SDE with drift $b=u+v$  and marginal density $\rho_t$ and let $\mathbb{W}$ denotes the Wiener law. Then we have that
  \begin{align*}\label{relative_entropy}
    H(\PP|\mathbb{W}) &= \int_0^T \int |b(x,t)|^2 \rho(x,t) \dd x \dd t\\
                &= \int_0^T \int_{\RR^d} \Big(u(x,t)^2 + v(x,t)^2\Big)\rho(x,t)\dd x\dd t + H(\rho(0)) - H(\rho(T)).
  \end{align*}
  where 
  $H$ denotes the Boltzmann entropy as given in Definition~\ref{def_boltzmann}.
\end{lemma}
\begin{proof}
The proof follows by straight-forward computations using the continuity equation, the definition of osmotic velocity, and Definition~\ref{def_boltzmann}. See, e.g, \cite{BEC2025}. For a more abstract setting see \cite{Cattiaux2021TimeRO}.
\end{proof}
\begin{remark}\label{Follmer to Carlen}
The condition \eqref{eq:39} in the definition of conservative diffusions
   is clearly related to the condition on the drift $b\in L^2$ as in \eqref{L2_condition}.
   The precise relation between these conditions is clarified by Lemma \ref{lem:vu} above.
\end{remark}

Coming back to the $N$ interacting diffusion system under study, we recall an useful property for the Boltzmann entropy, which says that it can only decrease, 
(for details we refer to \cite{HauMis}), \cite[Lemma 5.1]{DeVecchiUgolini}) and a monotonicity property of the relative entropy between marginal laws. 
\begin{proposition}\label{pro:entropy-limsup}
   Let us assume that $\rho_{N,1}(t)\rightarrow\rho_\infty(t)$,
    in law, $t$ almost surely and that the initial distributions $\rho_{N}(0)$ satisfy for all $N$:
    \begin{align*}
      \int_{\RR^{d}} |x|^2 \rho^1_{N}(0,x) \dd x \le C_1,
    \end{align*}
    for a given constants $C_1<+\infty$. 
    Then 
    \begin{align*}
      H(\rho_\infty) \le \liminf_N \frac{1}{N} H(\rho_N) 
      .
    \end{align*}
\end{proposition}
\begin{proof}
See, e.g., \cite{BEC2025} and \cite{HauMis}.
\end{proof}

\begin{definition}\label{one-particle}
    Let us introduce the \textit{one-particle} or \textit{normalized relative entropy} $$\KL(\PP|\WW)=\frac{1}{N}H(\PP|\WW),$$ which is the right object to study in the infinite particle limit.
\end{definition}

\begin{proposition}\label{pro:rele_bound}
  Let $\PP_N$ denotes the law of the $N$ interacting diffusions and
  let $\PP_N^{(n)}$ be the push-forward of $\PP_N$ along the projection $\pi_n:\RR^N\rightarrow\RR^n$, $n<N$. Let us assume that $\rho_N(0)\rightarrow\rho_\infty(0)$ as $N \rightarrow \infty$. 
  Then, 
  \begin{align*}
    \KL(\PP_{N}^{(n)}|\WW^{\otimes n})\le \KL(\PP_N|\WW^{\otimes N})
  \end{align*}
  and
  \begin{align*}
      \limsup_{N\rightarrow\infty}
      \KL(\PP_{N}^{(n)}|\WW^{\otimes n})
\le \KL(\PP^{\otimes n}|\WW^{\otimes n})
.
  \end{align*}
\end{proposition} 
\begin{proof}
    See, e.g., \cite{DeVecchiUgolini} and \cite{HauMis}.
\end{proof}

\paragraph{Kac chaos and classical entropic chaos (at fixed time)}

In the standard formulation of Kac and entropic chaoticity one considers the coupled law at fixed time $t$. Let us consider a symmetric law $\rho_N(t)$ on $E^N$ (for $N$ particles), in the sense that they are invariant under permutation of the coordinates,
and a limit measure $\bar{\rho}$. The usual definition of Kac chaos is the following.

\begin{definition}[Kac chaos]
 We say that $\rho_N^{(n)}(t)$ is \emph{$\bar{\rho}(t)$-chaotic} if, for any fixed $k\geq 1$,
\[
\rho_{N}^{(k)}(t) \xrightarrow[N \to \infty]{} \bar{\rho}^{\otimes k}(t),
\]
where $\rho_{N}^{(n)}(t)$ denotes the $k-$th marginal of $\rho_N(t)$, $\bar\rho^{\otimes k}$ denotes the $k-$th product measure, and the convergence is understood in the sense of weak convergence of probability measures.   
\end{definition}

We stress that original Kac chaos property is a property concerning the marginal at a fixed time $t$ of the laws of the  process.\\

There exists a stronger notion of chaoticity, denominated entropy chaos, introduced by Carlen et al.\ in \cite{Carlen2010}.  The authors define entropically chaotic  a sequence of probability densities when
 the sequence of the one-particle marginal densities converges weakly and in addition the Boltzmann entropy converges.
\begin{definition}[Entropy chaos]\label{fixedtime_entropic_chaos}
We say that $\rho_N$ is $\rho-$entropic chaotic if for any fixed $t\in [0,T]$
\begin{enumerate}
    \item $\rho^1_N(t)\to \rho(t)$ weakly.
    \item $ \frac{1}{N}H(\rho_N(t)) \to H(\rho(t)).$
\end{enumerate} 
\end{definition}

The entropy chaos at fixed time $t$ in Definition \ref{fixedtime_entropic_chaos} is stronger that Kac's chaos  and  implies it due to the monotonicity property in Proposition \ref{pro:rele_bound}. See \cite{Jabin2018} for the derivation of a stability entropic inequality giving rise to a quantitative propagation of chaos and see also \cite{Lacker_2023}.  

\begin{remark}
Some authors (see, e.g., \cite{albeverio2020strong,cattiaux2024entropypathspaceapplication,Jabin2018,Lacker_2023}) call entropy chaos the following stronger chaoticity property
\[
\mathcal{H}\bigl(\rho_N(t) \mid \bar{\rho}^{\otimes N}(t)\bigr)=\frac{1}{N} H\bigl(\rho_N(t) \mid \bar{\rho}^{\otimes N}(t)\bigr)
\xrightarrow[N \to \infty]{} 0.
\]

 Since $H(\rho^{\otimes N}(t))=NH(\rho(t))$, the above convergence could seem
 equivalent to the entropic chaos in Definition \ref{fixedtime_entropic_chaos}, but this is not the case. 
 The two definitions are equivalent only under a regularity condition on the limit density $\rho.$ See the proof of Theorem \ref{Strong_entropic_chaos}.
\end{remark}

Finally we introduce a dynamical entropic notion of chaoticity \textit{on the path space}.

\paragraph{Entropy chaos on the path space}
We provide a definition of some different forms of  propagation of chaos on the path space $ \Omega := C([0,T];E)$, from Kac chaos to weak and strong entropy chaos. 
\begin{definition}[Chaoticity properties on path space]\label{chaoticity_defiition}
    Let $\mathbb{P}_N$ be a sequence of (symmetric) probability measures on $\Omega^{N}$ and let $\mathbb{P}$ be a probability measure on $\Omega$. Denote by $\mathbb{P}^{(k)}_N$ the $k-$th marginal measure of $\mathbb{P}_N$.
    \begin{enumerate}
        \item We say that $\mathbb{P}_N$ is $\mathbb{P}$-Kac chaotic if, for any $k\in \mathbb{N}$, $\mathbb{P}^{(k)}_N$ converges to $\mathbb{P}^{\otimes k}$ in distribution.
        \item We say that $\mathbb{P}_N$ is (weakly) $\mathbb{P}$-entropy chaotic (with respect to the Wiener measure) if  
        \[
        \lim_{N \rightarrow +\infty}\mathcal{H}\left(\mathbb{P}^{(k)}_N|\mathbb{W}^{\otimes k}\right) = \mathcal{H}\left(\mathbb{P}^{(k)}|\mathbb{W}^{\otimes k}\right).\] 
        \item We say that $\mathbb{P}_N$ is strongly $\mathbb{P}$-entropy chaotic if  \[\lim_{N \rightarrow +\infty}\mathcal{H}\left(\mathbb{P}^{(k)}_N|\mathbb{P}^{\otimes k}\right) =0.\]
        \end{enumerate}
\end{definition}

These notions of entropic chaos originate starting from the work of F\"ollmer \cite{Follmer_Wiener_space,follmer1988random} even though they have not been explicitly introduced there. 
They find application also in the theory of large deviations beginning with the work by Dawson and G\"artner \cite{Dawsont01041987}. For other contributions, which consider the gradient flow structure of entropy, are, e.g., \cite{budhiraja2012large,JordanOtto/S0036141096303359}.\\

 In \cite{ChristianLeonard2014} the entropy chaos on path space is introduced in connection with the celebrated Schr\"odinger problem. See \cite{cattiaux2024entropypathspaceapplication} for the application of the theory to singular McKean-Vlason diffusions. Furthermore, 
 the theory has been applied to stationary Bose Einstein condensation (BEC) in \cite{albeverio2020strong,albeverio2022mean,albeverio2017entropy,albeverio2022some}, and to general time dependent BEC in \cite{BEC2025}.

\section{Entropy chaos on path space}
In this section we first discuss our general stochastic model and then we propose a criterion for proving strong entropy chaos which is mainly based on a weak formulation of entropy chaos.

\subsection{A general interacting diffusion system}

We consider an $N$-dimensional interacting diffusion of the type
\begin{equation}\label{N-dimensionalSDE} 
  \begin{aligned}
    \dd X_N(t) &= b_N(X_N(t),t)\dd t + \dd W_N(t),
  \end{aligned}
\end{equation}
where $W_N$ is a $Nd$-dimensional Brownian motion, $W$ a $d$-dimensional one and
where the drift coefficient depends on all the $N$ process components. Without assuming that \eqref{N-dimensionalSDE} has all the properties of being a conservative diffusion, we focus here on some notable consequences of the square-integrability property of the drift coefficient 
$b_N \in L^2(\R^{Nd}\times [0,T],\mathbb P_N(t,dx)dt)$, or, explicitly,
\begin{equation}\label{L2cond}
    \mathbb{E}_{\mathbb{P}_N}\left[\int_0^T b_N^2(t)\dd t\right]=    \int_0^T\int_{\R^{Nd}} |b_N|^2\mathbb P_N(t,\dd x) \dd t < \infty,
\end{equation} 
where $\mathbb P_N(t)$ denotes the law of $X_N(t)$ for a fixed time $t\in[0,T]$.
\begin{proposition}\label{L2drift_consequences}
    Let us consider the system \eqref{N-dimensionalSDE} and let us denote by $\mathbb P_N(t)$ the law of $X_N(t)$ for a fixed $t\in[0,T]$. 
    Consider the initial distribution  $\nu(x)=\rho_N(x,0)\dd x$ with finite entropy, i.e. such that 
    \begin{equation}\label{finite_initial_entropy}
        H(\rho_N(0))<\infty. 
    \end{equation}
    Let us assume that \eqref{L2cond} holds, i.e. $|b|\in L^2(\R^{Nd}\times [0,T],\mathbb P_N(t,\dd x)\dd t)$.
    Moreover, assume the function $\Lambda(x)=\log(\max\{|x|,1\})$ also belongs to $L^2(\R^{Nd}\times [0,T],\mathbb P_N(t,\dd  x)\dd t)$.
    Then $\mathbb P_N(t,\dd x)$ is absolutely continuous with respect to the Lebesgue measure for every $t\in [0,T]$, 
    $\mathbb P_N(t,dx)=\rho_N(t,x)\dd x$,
    with probability density $\rho_N(\cdot,t) \in W^{1,1}(\R^{Nd}).$ Furthermore, the osmotic velocity satisfies:
    \begin{equation}\label{finite_osmotic_velocity}
     \int \int |u_N|^2 \rho_N(x,t) \dd x\dd t<\infty.   
    \end{equation}
    
\end{proposition}
\begin{proof}
To the SDE \eqref{N-dimensionalSDE} we can associate the Fokker-Planck-Kolmogorov equation starting from the initial measure $\nu$ which satisfies the finite entropy condition \eqref{finite_initial_entropy}. Since the diffusion coefficient is bounded and constant and the drift coefficient satisfies \eqref{L2cond}, by Theorem 7.4.1  of \cite{bogachev2022fokker} the thesis of the theorem holds. 
\end{proof}

\begin{remark}
If we assume that $\mathbb{P}_N(t) \in W_2(\R^{dN}) $, namely $\mathbb{P}_N(t)$  belongs to the Wasserstein space of measures admitting the second moment, and $\int_0^T \int_{\R^{dN}}|x_N|^2\mathbb{P}_N(t,\dd x_N)<+\infty,$ assumption $\Lambda(x) \in L^2(\R^{Nd}\times [0,T],\mathbb P_N(t,\dd x)\dd t)$ is automatically satisfied.
\end{remark}

\subsection{A criterion for proving entropy chaos}
We formulate a criterion for proving the strong entropy chaos on path space of the coupled law $\mathbb{P}_N$ w.r.t. to the tensorized law $\mathbb{P}^{\otimes N}.$ Some recent convergence results based on relative entropy methods are first briefly recalled (see \cite{DeVecchi_Rigoni2024}).\\

The first result gives a useful form of weak convergence of process drifts.
\begin{proposition}\label{pro:drifts}
  Consider a sequence of diffusion laws $\PP_N$, $N\in\NN$, on the same filtered space $(\Omega,\mathcal F_{[0,T]})$, where each $\PP_N$ is absolutely continuous with respect to the Wiener measure $\WW$ on $\Omega$,
  weakly converging to a probability measure $\tilde \PP$ with $\sup_{N\in\NN}\KL(\PP_N|\mathbb W)<+\infty$.
  Then, for any bounded continuous function $K:[0,T]\times\RR^{}\rightarrow\RR^{d} $ we have
  \begin{align*}
  \lim_{N\rightarrow\infty}   \EE_{\PP_N}\left[ \int_0^T  K(X(t),t) \cdot \dd X(t) \right]
  =
    \EE_{\tilde \PP}\left[ \int_0^T K( X(t),t)\cdot \dd X(t) \right]
    ,
  \end{align*}
  where $X(t), t\in [0,T]$ denotes the canonical process on $\Omega.$
\end{proposition}
\begin{proof}
The idea of the proof is to approximate the stochastic integral
$\int_0^T K(X(t),t) \dd X$ by a discrete sum over a time partition $\pi$ 
and then to exchange the limit $N\rightarrow\infty$ 
of $\PP_N\rightarrow\tilde\PP$
with the limit $|\pi|\rightarrow 0$ of the mash-size of the partition
going to zero.
The fact that these limits can be exchanged is established by
showing that the approximating discrete sum converges to the 
stochastic integral in $L^2(\PP_N)$ uniformly in $N$.
This is a consequence of the representation
$\dd X(t) = r(t)\dd t+ \dd W(t)$ 
which is possible since all the $\PP_N$ are assumed
absolutely continuous with respect to $\PP_W$
and the fact that $\sup_N \EE_{\PP_N}[\int_0^T |r(t)|^2\dd t]
=\sup_N\mathcal H(\PP_N|\PP_W)<\infty$.
The complete proof can be found in (\cite{DeVecchi_Rigoni2024}, Lemma 2.24).
\end{proof}

Since the weak convergence of Markovian probability laws do not necessarily converge to a Markovian limit law, we need of the following result.

\begin{proposition}\label{pro:time_proj}
  Let $\PP$, a probability measure on the probability space $\Omega=C([0,T];\RR^d),$
  be the law of an $\RR^d$-valued stochastic process $X$ with continuous paths
  and let $\mu_t=\Law(X_t)$, $t\in[0,T]$ be the marginal law at time $t$ of the process $X$,
  i.e. the push forward of $\PP$ under the measurable map which sends each path $\omega\in\Omega$
  to its value $\omega(t)$ at time $t$.
  Consider a measurable function $\mathcal D:\Omega \times [0,T]\rightarrow\RR$
  such that $\EE[ \int_0^T |\mathcal D(X_t,t)|^2 \dd t ] <+\infty$.
  Then there exists a measurable function $D:\RR^n\times [0,T]\rightarrow\RR$
  that satisfies the following properties.
  \begin{enumerate}
  \item $\int_0^T\int_{\RR^d} |D(x,t)|^2\mu_t(\dd x)\dd t < +\infty$;
  \item For any $f\in C_b([0,t]\times\RR^d)$ it satisfies
    \begin{equation*}
      \int_0^T \int_{\RR^n} D(x,t)f(x,t) \mu_t(\dd x)\dd t
      = \EE\left[ \int_0^T \mathcal D(\cdot,t) f(X_t(\cdot),t)\dd t \right] \dd t
      .
    \end{equation*}

    In particular we have
    \begin{equation*}
      \int_0^T \int_{\RR^n} |D(x,t)|^2 \mu_t(\dd x)\dd t
      \leq \EE_\PP\left[ \int_0^T |\mathcal D(\cdot,t)|^2 \dd t \right]
      ,
    \end{equation*}
    where the equality holds if and only if $\mathcal D(\cdot,t) = D(X_t(\cdot),t)$.
  \end{enumerate}  
\end{proposition}
\begin{proof}
The idea of the proof is to identify functions
$L^2(\RR^n\times[0,T],\mu_t\dd t)$ as elements of $L^2(\PP)\otimes L^2([0,T])$
by the map $T$ which sends $f(x,t)\mapsto f(X(t),t)$. 
The map $T$ is a linear isometry and therefore is injective and its range $\tilde{\mathcal L}$ is 
a Hilbert subspace of $L^2(\PP)\otimes L^2([0,T])$.
Finally one takes $D= T^{-1} P_{\tilde{\mathcal L}} \mathcal D$
where $P_{\tilde{\mathcal L}}$ denotes the orthogonal projection
from $L^2(\PP)\otimes L^2([0,T])$ onto $\tilde{\mathcal L}$.
For the detailed proof see (\cite{DeVecchi_Rigoni2024}, Lemma 2.15).
\end{proof}

Our criterion is based on some results in \cite{DeVecchi_Rigoni2024}, presented in a way which is useful to our goal and which has been proved in \cite{BEC2025}. From the strong entropic chaos introduced in Definition \ref{chaoticity_defiition}  is trivial that the sequence is also Kac chaotic (on the path space).
The next theorem states under which conditions this weaker entropic notion implies the Kac chaos on the path space.
\begin{theorem}\label{Th:Criterium}
  Let $X_N$, $N\in\NN$, denote a family of $d-$dimensional diffusions,
  which are solutions to $N$-dimensional SDEs driven by Brownian motion,  with drifts $b_N$ and law $\PP_N$.
  Moreover, let $X$ be another $d$-dimensional diffusion,
  solution to a Brownian SDE with drift $b$ and law $\PP$.
  We assume the following conditions hold: 
  \begin{enumerate}[(i)]
  \item $ H(\PP_{N}| \mathbb W)<+\infty$, for all $N\in\NN$
    and $ H(\PP| \mathbb W)<+\infty$ and
     \begin{align*}\label{limsup_control}
      \limsup_{N \rightarrow +\infty} \KL(\PP_{N}| \mathbb W^{\otimes N}) \leq  H(\PP| \mathbb W).
     \end{align*}
    \item The marginals $\mu^N_t$ of $X_N(t)$, $N\in\NN$, and $\mu_t$ of $X(t)$
      exist $\forall t \in [0,T]$ and satisfy
      $\lim_{N\rightarrow\infty} \mu^N_t = \mu_t$, $\forall t \in [0,T]$,
      weakly for $t$ Lebesgue almost surely.
  \item For every $K: \RR^d\times [0,T]\rightarrow \RR$ continuous bounded function,
  the expectations
    $\EE[\int_0^T K(Y_N(t),t) \dd X_N(t)]$ and
    $\EE[\int_0^\infty K(X(t),t) \dd X(t)]$ 
    are well defined and we have that
    \begin{align*}
      \lim_{N\rightarrow\infty} \EE\left[\int_0^T K(X_N(t),t) \cdot \dd X_N(t)\right] = \EE\left[\int_0^T K(X(t),t) \cdot \dd X(t)\right].
    \end{align*}
  \end{enumerate}
      Then, for any $k \in \mathbb{N}$, denoting by $\PP_{N}^{(k)}$ the projection of the measure $\PP_N$ on the first $k$ particles, we have that $\PP_N^{(k)}$ converges to $\PP^{\otimes k}$ weakly.
\end{theorem}
\begin{proof}
The idea of the proof is to use the
property of compactness of the level sets of the relative entropy with respect to the
topology of weak convergence of probability measures.
Then by condition $(i)$
one can pass to a convergent subsequence.
Then one can use F\"olmer result
which relates the relative entropy with $L^2$ norm of the drift.
Unfortunately Markovianity is not
preserved by weak limits of the law of the process.
Therefore one needs to use 
Proposition~\ref{pro:drifts} 
and Proposition~\ref{pro:time_proj}
to prove that, for the convergent subsequences of laws,
the limit laws are indeed a weak solution
of the same SDE.
Then by uniqueness of weak solution one obtains that the whole sequence. 
    For the full proof, in the case of the $k=1$, we refer to  Theorem 5.4 of \cite{BEC2025}. 
    See, also, Section 7.1 of \cite{DeVecchi_Rigoni2024} for other very similar results.
\end{proof}

\begin{remark}\label{rem:weak_entropic_chaos}
\begin{enumerate}[(i)]
    \item Since the relative entropy is lower semi-continuous under weak convergence of the probability measures, then assumption (i)
together with (iii) implies that
    \begin{align*}
      \lim_{N\rightarrow \infty} \mathcal H(\PP_{N}| \mathbb W^{\otimes N}) = \mathcal H(\PP| \mathbb W)
      ,
    \end{align*}
    which is the weak entropic chaos introduced in Definition \ref{Strong_entropic_chaos}.
\item We stress that we do not require any regularity assumptions on the drift coefficient, only the relative entropy boundedness condition in (i), implying, by Girsanov theory, the $L^2-$type condition \eqref{L2_condition_intro}.
\item The hypothesis (ii) and (iii) have to be proved according to the specific model at hand. 
In the time-dependent BEC we use a compactness argument and energy bounds \cite{BEC2025}. See also Section 5.
   \end{enumerate} 
\end{remark}
\begin{remark}
    The theorem is a sort of generalization to the path-space setting of point (iii) and (iv) in (\cite{HauMis}, Theorem 1.4).
\end{remark}
We finally provide the condition under which the strong entropic chaos holds in our general stochastic model.
\begin{theorem}\label{Strong_entropic_chaos}
Suppose that in addition to the hypotheses of Theorem \ref{Th:Criterium} we have that the limit drift $b$ is continuous and, for any $t \in [0,T],$ $|b(x,t)| \lesssim |x|^2+1$. Then if $\mathbb{P}_N$ is weakly $\mathbb{P}$-entropy chaotic it is also strongly $\mathbb{P}$-entropy chaotic.  
\end{theorem} 

\begin{proof}
Since $ \mathbb{P}_N$ is the law of the solution $X_N$ of the SDE \eqref{eq:6} and $\mathbb{P}$ is the law of the solution $X$ to the isotropic SDE having drift $b$,  by  Girsanov theory we get
\begin{align*} 
&\mathcal{H}(\mathbb{P}_N|\mathbb{P}^{\otimes N})
= \int_0^T\int_{\R^{dN}}
\left|b_N(x_1,...,x_N,t)-b(x_1,t)\right|^2 
\rho_N(\dd x_1,...,\dd x_N,t)\dd t
\\
&= \int_0^T\int_{\R^{dN}}
|b_N|^2\rho_N(\dd x_1,...,\dd x_n,t)\dd t 
- 2\int_0^T\int_{\R^{dN}} b(x_1,t)\cdot b_N(x_1,...,x_N,t) 
\rho_{N}(\dd x_1,...,\dd x_N,t)\dd t \;+ 
\\
&\hspace{26em}+ \int_{0}^T|b(x_1,t)|^2\rho_{N}(\dd x_1,...,\dd x_N,t)\dd t.
\end{align*}
Since $b$ is continuous and grows at most quadratically, then the second term converges by assumption (iii) in Theorem \ref{Th:Criterium} and the third term convergences by weak convergence of the marginal laws, that is, 
assumption (ii) in Theorem \ref{Th:Criterium}. 
Finally, the convergence of $\mathcal{H}(\mathbb{P}_N|\mathbb{P}^{\otimes N})$
to zero  is guaranteed by the weak entropic chaos deriving by assumption (i) in Theorem \ref{Th:Criterium} (as explained in Remark \ref{rem:weak_entropic_chaos}). Indeed, by Girsanov theory, the weak entropic chaos is equivalent to the weak $L^2$ convergence of the drifts $b_N$.
\end{proof}

Strong entropic chaos is the strongest notion, due to the monotonicity property of relative entropy in Proposition \ref{pro:rele_bound}, because it controls all the time correlations and therefore it implies both entropic chaos for any fixed $t$, and the chaos property for the finite dimensional distributions. We underline that the introduction  of the weak entropic chaos allows to formulate a criterion (Theorem \ref{Strong_entropic_chaos}) for proving strong entropic chaos on the path space with only a regularity property on the limit drift $b$ and no regularity conditions on the sequence of drifts $b_N$, except the $L^2-$type condition \eqref{L2_condition_intro}.

\section{Applications}
The interacting diffusion system \eqref{N-dimensionalSDE} is very general because of the weak condition on the drift coefficient \eqref{L2cond}. To prove strong entropic chaos  of the involved coupled laws w.r.t. the limit low at the process level, one can apply our Theorem \ref{Th:Criterium} and then, if some additional regularity conditions on the limit drift are satisfied, use (Theorem \ref{Strong_entropic_chaos}).\\

In \cite{BEC2025}, we employed the above criterion to show propagation of chaos on the path-space 
for a mean-field \emph{time-dependent} model of Bose Einstein Condensation (BEC).\\

The starting point of \cite{BEC2025} is to associate, via the Nelson-Carlen map, 
an interacting $N$-particle system of conservative diffusions to the (linear) Sch\"odinger equation for 
a system of $N$ Bose particles.
More precisely, the drift is expressed as a gradient $b_N(x,t)=\frac{\nabla \Psi_N(x,t)}{\Psi_N(x,t)}$ in terms of 
the solution $\Psi_N(x^N,t)$ of the $N$-particle Schr\"odinger equation. 
Since the solution to the Schr\"odinger equation $\Psi_N(x^N,t)$ is expected to have nodes, 
the drift can be  very singular. In particular, because of the $L^2$ structure of the quantum theory,
one can expect only that the condition \eqref{L2_condition_intro} is satisfied by the drift. 
We underline that, because of the quantum origin of the problem, 
the usual assumption of boundness and Lipschitzianity of the drift cannot be assumed in general.
Denoting by $\mathbb{P}_N $ the law of the process $X_N$ on the path-space $\Omega =C([0,T],\R^{Nd})$,
and introducing the conservative  diffusion $X$ having law $ \PP$ corresponding to the one-particle
nonlinear Schr\"odinger equation, we are exactly in the probabilistic framework described in the present paper.\\

In order to prove Kac chaoticity of the  joint-law $\mathbb{P}_N$ w.r.t. the limit law $\mathbb{P}$ in \cite{BEC2025},
one can use Theorem \ref{Th:Criterium}. 
The finite energy condition \eqref{finite_condition} corresponds, in the quantum mechanical setting,
 to finite quantum kinetic energy.
Because of the quantum-mechanical conservation of total (kinetic plus potential) energy,
one can convert this bound on the kinetic energy to a bound on the potential energy, 
which is easier to handle.
The convergence of the marginal densities can be established by compactness and by uniqueness of the solution to the limit nonlinear Schr\"odinger equation. The weak convergence of the drift sequence has been showed 
by detailed consideration of the regularity properties of some associated of trace class operators (\cite{BEC2025}). 
Therefore, after application of Theorem \ref{Th:Criterium}, one obtains the Kac chaos property on path space
for the mean-field limit. \\

Differently from the time-dependent setting, 
in the stochastic description of \emph{stationary} BEC one can take full advantage of the convexity property of the quantum energy functional. It is well-known that there are three types of rescaling of the interaction potential characterized by a parameter $0\leq \beta < 1,$ modeling the $N-$dependence of the range of the potential.  
In the stochastic approach to stationary BEC, not only the mean-field case ($\beta=0)$,
but also the general mean-field or intermediate scaling limit ($ 0< \beta < 1$),
as well as the more complex non mean-field case ($\beta=1$), known as Gross-Pitaevskii scaling limit, 
has been analyzed by entropy methods. 
By these methods, one obtains strong convergence results on the path space in the general mean-field setting \cite{albeverio2020strong}, and a weak convergence result in the Gross-Pitaevskii one \cite{albeverio2017entropy}. 
In all cases the crucial point has been first to prove the weak $L^1$ convergence of the one-particle quantum energy functional to the limit one associated with the limit nonlinear Schr\"odinger equation. 
The energy functional consists of two parts: 
a kinetic term, which is quadratic in the drift, and a potential term, 
consisting of the sum of a confining and an interaction potential. 
From the convergence of the energy functional, the weak convergence of the fixed time marginal densities can be derived. 
The second key issue is if the kinetic and the potential parts converge individually or not. 
In the general mean-field case, the convergence of, separately, both the kinetic term and the potential one (\cite{albeverio2020strong}) has two important consequences: 
first the Fisher information chaos holds for the fixed time marginal densities (\cite{albeverio2020strong}, Proposition 6.1) and, second, Theorem \ref{Strong_entropic_chaos} can be applied because the minimizer of the nonlinear functional is strictly positive and it is continuous, thus allowing to obtain the strong entropy chaos on path space. 
Since the latter implies the strong entropy chaos for the marginal laws (see, e.g. \cite{albeverio2020strong}, Lemma 5.1), 
the total variation convergence of the marginal laws was obtained. \\

The more difficult stationary Gross-Pitaevskii scaling limit, which does not have any mean-field character, 
as we have already mentioned, can be also handled by an entropy approach (\cite{albeverio2017entropy,DeVecchiUgolini,morato2011stochastic,morato2013localization,Ugolini}). 
Even though the relative entropy sequence does not go to zero in that setting, 
fortunately it happens to converge to a finite limit. 
This in particular states the non convergence of the Fisher information sequence, 
implying that Fisher chaos does not hold \cite{albeverio2017entropy}. 
Nevertheless, under a suitable regularity condition on the confining potential,
a strong form of the usual propagation of chaos or Kac chaos is achieved, 
where the convergence of the fixed time marginal densities is understood in the strong $L^1$ sense.   
As a consequence of the non convergence of the Fisher information sequence, 
only a weak convergence of the process marginal laws on path space can be proved by entropy methods, 
using some non-trivial properties of stopping times, to prevent the process from entering in the interacting regions are crucial to obtain the  result \cite{albeverio2017entropy}.\\

The probabilistic setting, originating from the of stationary BEC, 
has been used to propose a novel convergence scheme
in the context of ergodic stochastic optimal control \cite{albeverio2022some,albeverio2022mean}. 
The approach used in the previous two cited papers has been further generalized in \cite{DeVecchi_Rigoni2024} to the case of stochastic optimal control problems of  Schr\"odinger or finite horizon form. Indeed, in \cite{DeVecchi_Rigoni2024}, the convergence of $N$-particle stochastic control problems to their McKean-Vlasov limits has been established. This result is  based on a reformulation of the control problem via a Benamou-Brenier fluid-dynamic approach using a superposition principle coming from optimal transport theory. 
The authors demonstrate that the optimal $N$-particle laws are weak and strong entropy chaotic in the sense described in Definition \ref{chaoticity_defiition}, obtaining, thus, that these probability laws converge to the solution of the limit McKean-Vlasov  problem.\\

We note that the model here analyzed includes also the usual mean field  form of the drift given by $b_N(x_N,t)=\frac{1}{N}\sum_{j=1}^N \nabla \Phi (X_i-X_j),$ with $$ \mathbb{E}\left[\int_0^T \int_{\R^d} \int_{\R^d} |\nabla \Phi (x-y)|^2d\mu_t(dx)d\mu_t(dy)\right] <\infty,$$ 
where $\mu_t$ denotes the marginal law of the single component $X_j$ at time $t.$ Finally, all the interacting diffusion systems in which the drift depends on all the $N$ variables are included in our stochastic scheme.

\bibliography{references} 
\bibliographystyle{abbrv}

\end{document}